\documentclass[12pt,reqno]{amsart}

\usepackage[latin1]{inputenc}
\usepackage{amsmath}
\usepackage{amsfonts}
\usepackage{amssymb}
\usepackage{graphics}
\usepackage{enumerate}
\usepackage{subfigure}
\usepackage{amssymb,amsmath,amsthm,amscd,epsf,latexsym,verbatim,graphicx,amsfonts}
\input epsf.tex

\usepackage{amscd}
\usepackage{amsmath}
\usepackage{amssymb}
\usepackage{amsthm}
\usepackage{epsf}
\usepackage{latexsym}
\usepackage{verbatim}
\input epsf.tex

\theoremstyle{plain}
\newtheorem{theorem}{Theorem}[section]

\newtheorem{corollary}[theorem]{Corollary}
\newtheorem{lemma}[theorem]{Lemma}
\newtheorem{prop}[theorem]{Proposition}
\theoremstyle{definition}

\theoremstyle{remark}

\newcommand{\nri}{n\rightarrow\infty}

\newcommand{\bbR}{\mathbb{R}}
\newcommand{\bbC}{\mathbb{C}}

\newcommand{\bbD}{\mathbb{D}}
\newcommand{\bbJ}{\mathbb{J}}
\newcommand{\bbN}{\mathbb{N}}

\newcommand{\eitheta}{e^{i\theta}}
\newcommand{\tha}{\theta_{\alpha}}

\DeclareMathOperator*{\supp}{supp}
\DeclareMathOperator*{\Real}{Re}
\DeclareMathOperator*{\Imag}{Im}

\title[]{Universality at an Endpoint for Orthogonal Polynomials with Geronimus-Type Weights}
\author[]{Brian Simanek}
\date{}

\begin{document}
\maketitle

\begin{abstract}
We provide a new closed form expression for the Geronimus polynomials on the unit circle and use it to obtain new results and formulas.  Among our results is a universality result at an endpoint of an arc for polynomials orthogonal with respect to a Geronimus type weight on an arc of the unit circle.  The key tool is a formula of McLaughlin for the $n^{th}$ power of a $2\times2$ matrix, which we use to derive convenient formulas for Geronimus polynomials.
\end{abstract}

\vspace{4mm}


\vspace{2mm}


\vspace{2mm}

\normalsize

\section{Introduction}\label{intro}

Let $\mu$ be a probability measure whose support is an infinite and compact subset of the unit circle $\partial\bbD$ in the complex plane.  Let $\{\Phi_n(z;\mu)\}_{n=0}^{\infty}$ be the sequence of monic orthogonal polynomials for the measure $\mu$ and let $\{\varphi_n(z;\mu)\}_{n=0}^{\infty}$ be the sequence of orthonormal polynomials.  It is well-known that corresponding to this measure is a sequence of Verblunsky coefficients $\{\alpha_n\}_{n=0}^{\infty}\in\bbD^{\bbN_0}$ so that
\begin{align}\label{sr}
\Phi_{n+1}(z;\mu)=z\Phi_n(z;\mu)-\bar{\alpha}_n\Phi_n^*(z;\mu),\qquad n\in\bbN_0,
\end{align}
where $\Phi_n^*(z;\mu):=z^n\overline{\Phi_n(1/\bar{z};\mu)}$.  The formula (\ref{sr}) is often called the \textit{Szeg\H{o} recursion} (see \cite[Section 1.5]{OPUC1}).  The relationship between infinitely supported probability measures on the unit circle and sequences of Verblunsky coefficients is a bijection (see \cite[Section 1.7]{OPUC1}) and there is a substantial literature describing the relationship between the sequence and the corresponding measure (see \cite{OPUC1,OPUC2} and references therein).

Our focus in this work will be on the so-called \textit{Geronimus polynomials}, which are orthogonal with respect to the measure corresponding to the sequence of Verblunsky coefficients $\{\alpha,\alpha,\alpha,\ldots\}$ for some $\alpha\in\bbD$.  The measure of orthogonality in this case is supported on an arc of the unit circle whose length depends on $\alpha$ and possibly a mass point outside this arc, whose weight depends on $\alpha$.  The Geronimus polynomials have been studied before (see \cite{Golinskii,GNPV,GNV,BHMD,Cargile,Pinter}) and there is a known closed form expression for them (see also \cite[Section 1.6]{OPUC1}).  This formula was later used by Lubinsky and Nguyen in \cite{LubNg} to obtain a universality result for certain polynomial reproducing kernels at an interior point of the arc supporting the measure of orthogonality.  Our goal will be to provide a new closed form expression for the Geronimus polynomials, which will enable us to prove several new results and formulas, including a universality result at the endpoint of the arc supporting the measure of orthogonality.

The key tool in our analysis comes from matrix theory.  The Szeg\H{o} recursion  can be written
\[
\begin{pmatrix}
\Phi_{n+1}(z;\mu)\\ \Phi_{n+1}^*(z;\mu)
\end{pmatrix}=
\begin{pmatrix}
z & -\bar{\alpha}_n\\ -\alpha_nz & 1
\end{pmatrix}
\begin{pmatrix}
\Phi_{n}(z;\mu)\\ \Phi_{n}^*(z;\mu)
\end{pmatrix},
\]
(see \cite[Section 3.2]{OPUC1}).  The $2\times2$ matrix in this relation is called the $n^{th}$ \textit{transfer matrix} for $\mu$.  If the Verblunsky coefficients form a constant sequence, then one can recover the polynomial $\Phi_n(z;\mu)$ in a straightforward way by using the following formula for the $n^{th}$ power of a $2\times2$ matrix.

\begin{theorem}[McLaughlin, \cite{McL}]\label{wmain}
Let $A$ be a $2\times 2$ matrix given by
\[
A=\begin{pmatrix}a & b\\ c & d\end{pmatrix}.
\]
If $R$ denotes the trace of $A$ and $D$ denotes its determinant, then
\[
A^n=\begin{pmatrix}y_n-dy_{n-1} & by_{n-1}\\ cy_{n-1} & y_n-ay_{n-1}\end{pmatrix},
\]
where
\begin{align}\label{ydef}
y_n=\sum_{m=0}^{\left\lfloor\frac{n}{2}\right\rfloor}\binom{n-m}{m}R^{n-2m}(-D)^m.
\end{align}
\end{theorem}

This simple result is all that we require to prove our new formula, which appears as Theorem \ref{geronclosed}.  Before we can state our results and formulas in Section \ref{geronsec}, we review some notation and terminology in the next section.  Finally, in Section \ref{univ}, we state and prove our universality result.

\section{Preliminaries}\label{prelim}

In this section we discuss some notation, formulas, and terminology that we will use throughout Sections \ref{geronsec} and \ref{univ}.  Many of the topics we discuss here are part of a rich theory that is too long to discuss in full detail.  Therefore, we will focus only on the specific formulas that we will need for our proofs.

\subsection{Chebyshev Polynomials}\label{chebsec}

The formula that we will obtain for the Geronimus polynomials involves the Chebyshev polynomials of the second kind, which are orthonormal with respect to the measure $\frac{2}{\pi}\sqrt{1-x^2}dx$ on the interval $[-1,1]$.  We denote this sequence of polynomials by $\{U_n\}_{n=0}^{\infty}$ and note that these polynomials are given by the formula
\begin{equation}\label{chebform1}
U_n(x)=\sum_{j=0}^{\left\lfloor\frac{n}{2}\right\rfloor}(-1)^j\binom{n-j}{j}(2x)^{n-2k}
\end{equation}
(see \cite[page 37]{BE}).  We also recall from \cite[page 37]{BE} that
\begin{equation}\label{chebform2}
U_n(x)=\frac{(x+\sqrt{x^2-1})^{n+1}-(x-\sqrt{x^2-1})^{n+1}}{2\sqrt{x^2-1}}
\end{equation}

We will make one use of the Chebyshev polynomials of the first kind, which are orthogonal with respect to the measure $\frac{1}{\pi\sqrt{1-x^2}}dx$ on the interval $[-1,1]$.  We will denote this sequence of polynomials by $\{T_n\}_{n=0}^{\infty}$ and define them by the formula
\begin{equation}\label{TU}
T_n(x)=U_n(x)-xU_{n-1}(x)
\end{equation}
(see \cite[page 37]{BE}).

\subsection{Second Kind Polynomials}\label{sk}

We have already mentioned that to every sequence of complex numbers $\{\alpha_0,\alpha_1,\alpha_2,\ldots\}\in\bbD^{\bbN_0}$ there corresponds a unique probability measure $\mu$ on the unit circle having infinite support.  The sequence $\{\alpha_n\}_{n=0}^{\infty}$ generates the sequence of orthogonal polynomials $\{\Phi_n(z;\mu)\}_{n=0}^{\infty}$ via the Szeg\H{o} recursion.  One can similarly generate a sequence of monic polynomials from the Szeg\H{o} recursion using the sequence $\{-\alpha_0,-\alpha_1,-\alpha_2,\ldots\}$, and the resulting polynomials are what we call the \textit{second kind polynomials} for the measure $\mu$ and we denote them by $\{\Psi_n(z;\mu)\}_{n=0}^{\infty}$ as in \cite{OPUC1,OPUC2}.  The polynomials $\{\Psi_n(z;\mu)\}_{n=0}^{\infty}$ are also orthogonal with respect to a probability measure on the unit circle, which is in the family of Aleksandrov measures for the measure $\mu$.  We will not make use of this particular fact, so we refer the reader to \cite[Section 1.3.9]{OPUC1} for details.

\subsection{Wall Polynomials}\label{wallsec}

Corresponding to every probability measure on the unit circle is a Schur function $f$, which maps $\bbD$ to itself.  When the support of $\mu$ is finite, this map is a Blaschke product, but when the support is infinite, there is a canonical pair of sequences of polynomials $\{A_n\}_{n=0}^{\infty}$ and $\{B_n\}_{n=0}^{\infty}$ such that $A_n/B_n$ converges to $f$ uniformly on compact subsets of $\bbD$ as $\nri$ (see \cite[Section 1.3.8]{OPUC1}).  These polynomials are called the \textit{Wall polynomials} for the measure $\mu$ and the Pint\'{e}r-Nevai formulas (see \cite[Theorem 3.2.10]{OPUC1} or \cite{PN}) tell us that
\begin{align*}
A_n(z)&=\frac{\Psi_{n+1}^*(z;\mu)-\Phi_{n+1}^*(z;\mu)}{2z}\\
B_n(z)&=\frac{\Psi_{n+1}^*(z;\mu)+\Phi_{n+1}^*(z;\mu)}{2}
\end{align*}

\subsection{Paraorthogonal Polynomials}\label{popuc}

Suppose $\mu$ is a probability measure on the unit circle having $\{\Phi_n(z;\mu)\}_{n=0}^{\infty}$ as its monic orthogonal polynomials.  For each $\beta\in\partial\bbD$ and each $n\in\bbN_0$, one defines the paraorthogonal polynomial $\Phi_{n+1}^{(\beta)}(z;\mu)$ by
\[
\Phi_{n+1}^{(\beta)}(z;\mu):=z\Phi_n(z;\mu)-\bar{\beta}\Phi_n^*(z;\mu).
\]
We also define
\[
\Psi_{n+1}^{(\beta)}(z;\mu):=z\Psi_n(z;\mu)-\bar{\beta}\Psi_n^*(z;\mu)
\]
for each $n\in\bbN_0$.  Paraorthogonal polynomials were introduced in \cite{JNT} and have the property that all of their zeros are simple and lie on the unit circle.  Paraorthogonal polynomials arising from Geronimus polynomials have been previously considered in \cite{CLMSR}.

\subsection{Regularity}\label{reg}

If $\mu$ is a probability measure on the unit circle with orthonormal polynomials $\{\varphi_n(z;\mu)\}_{n=0}^{\infty}$, let $\kappa_n$ denote the leading coefficient of $\varphi_n$.  Following the terminology from \cite{StaTo}, we will say that the measure $\mu$ is \textit{regular} if
\[
\lim_{\nri}\kappa_n^{1/n}=\frac{1}{\mbox{cap}(\supp(\mu))},
\]
where $\mbox{cap}(K)$ is the logarithmic capacity of the compact set $K$.  Regularity is a complicated notion and we will not discuss the technical details here.  We mention that a measure $\mu$ whose support is an arc $\Gamma$ of the unit circle is regular if and only if
\begin{equation}\label{regprop}
\lim_{\nri}\left(\sup_{{\deg(P)\leq n}\atop{P\not\equiv0}}\left[\frac{\|P\|_{L^{\infty}(\Gamma)}}{\|P\|_{L^2(\mu)}}\right]^{1/n}\right)=1
\end{equation}
(see \cite[Theorem 3.2.3(v)]{StaTo}).  Regularity indicates that the measure $\mu$ has sufficient density that a polynomial cannot have an exponentially small $L^2$-norm without having an exponentially small $L^{\infty}$-norm.

\medskip

With these preliminaries in hand, we can now proceed to state and prove our new results.

\section{Geronimus Polynomials}\label{geronsec}


For any $\alpha\in\bbD$, let $\rho=\sqrt{1-|\alpha|^2}$, and let $\mu_{\alpha}$ be the probability measure on the unit circle whose Verblunsky coefficients satisfy $\alpha_n=\alpha$ for all $n\in\bbN_0$.  This measure is supported on the arc $\{\eitheta:2\arcsin(|\alpha|)\leq\theta\leq2\pi-2\arcsin(|\alpha|)\}$ and possibly one point outside this arc (see \cite[Section 1.6]{OPUC1}).  Our first result is a new formula for the polynomials $\varphi_n(z;\mu_{\alpha})$ and $\varphi_n^*(z;\mu_{\alpha})$.

\begin{theorem}\label{geronclosed}
For any $\alpha\in\bbD$ and $n\in\bbN_0$, it holds that
\begin{align*}
\varphi_n(z;\mu_{\alpha})&=z^{n/2}\left(U_n\left(\frac{z+1}{2\rho\sqrt{z}}\right)-\frac{1+\bar{\alpha}}{\rho\sqrt{z}}U_{n-1}\left(\frac{z+1}{2\rho\sqrt{z}}\right)\right)\\
\varphi_n^*(z;\mu_{\alpha})&=z^{n/2}\left(U_n\left(\frac{z+1}{2\rho\sqrt{z}}\right)-\frac{\sqrt{z}(1+\alpha)}{\rho}U_{n-1}\left(\frac{z+1}{2\rho\sqrt{z}}\right)\right)
\end{align*}
where $U_{-1}=0$.
\end{theorem}

\begin{proof}
Since the Verblunsky coefficients for the Geronimus polynomials are all the same, we have
\[
\begin{pmatrix}
\Phi_{n}(z;\mu_{\alpha})\\ \Phi_{n}^*(z;\mu_{\alpha})
\end{pmatrix}=
\begin{pmatrix}
z & -\bar{\alpha}\\ -\alpha z & 1
\end{pmatrix}^n
\begin{pmatrix}
1\\ 1
\end{pmatrix},
\]
Therefore, Theorem \ref{wmain} implies
\begin{align}
\label{phident}\Phi_n(z;\mu_{\alpha})&=y_n-(1+\bar{\alpha})y_{n-1},\\
\label{phisdent}\Phi_n^*(z;\mu_{\alpha})&=y_n-z(1+\alpha)y_{n-1},
\end{align}
where for any choice of $\sqrt{z}$ we have
\[
y_n(z)=\sum_{m=0}^{\left\lfloor\frac{n}{2}\right\rfloor}\binom{n-m}{m}(z+1)^{n-2m}(-\rho^2z)^m=\rho^nz^{n/2}U_n\left(\frac{z+1}{2\rho\sqrt{z}}\right).
\]
If we plug this into (\ref{phident}) and (\ref{phisdent}) and note that the leading coefficient of $\varphi_n$ is $\rho^{-n}$ (see \cite[Equation 1.5.22]{OPUC1}), we get the desired formulas.
\end{proof}

Let us explore some elementary consequences of Theorem \ref{geronclosed}.  First notice that we can find the generating function for the polynomials $\{\varphi_n(z;\mu_{\alpha})\}_{n=0}^{\infty}$.

\begin{corollary}\label{genfunc}
The polynomials $\{\varphi_n(z;\mu_{\alpha})\}_{n=0}^{\infty}$ satisfy
\[
\sum_{n=0}^{\infty}\varphi_n(z;\mu_{\alpha})t^n=\frac{\rho-t-t\bar{\alpha}}{\rho-t(z+1)+\rho zt^2}
\]
whenever this series converges.
\end{corollary}

\begin{proof}
This is an immediate consequence of Theorem \ref{geronclosed} and the fact that
\[
\sum_{n=0}^{\infty}U_n(x)t^n=\frac{1}{1-2xt+t^2}
\]
whenever this series converges (see \cite[Equation 4.5.23]{Ibook}).
\end{proof}

As a second application, we can use Theorem \ref{geronclosed} to find convenient formulas for the Wall polynomials for the measure $\mu_{\alpha}$.

\begin{corollary}\label{wall}
For all $n\in\bbN$, the Wall polynomials $A_n$ and $B_n$ for the measure $\mu_{\alpha}$ are given by
\begin{align*}
A_n(z)&=\alpha\rho^{n}z^{n/2} U_n\left(\frac{z+1}{2\rho\sqrt{z}}\right)\\
B_n(z)&=\rho^{n+1}z^{(n+1)/2} \left[U_{n+1}\left(\frac{z+1}{2\rho\sqrt{z}}\right)-\frac{\sqrt{z}}{\rho}U_{n}\left(\frac{z+1}{2\rho\sqrt{z}}\right)\right]
\end{align*}
\end{corollary}

\begin{proof}
This is an immediate consequence of Theorem \ref{geronclosed} and the Pint\'{e}r-Nevai formulas.
\end{proof}

As an additional application, we have the following relation for the first and second-kind paraorthogonal polynomials.

\begin{corollary}\label{geronone}
For every $\alpha\in\bbD$ and every $n\in\bbN$ it holds that
\[
\Phi_n^{(1)}(1;\mu_{\alpha})+\Psi_n^{(1)}(1;\mu_{\alpha})=0.
\]
\end{corollary}

We can also use Theorem \ref{geronclosed} to provide new proofs of some existing results.  For instance, we can apply Corollary \ref{wall} and send $\nri$ to find the Schur function for the measure $\mu_{\alpha}$.    Indeed, by \cite[Theorem 1]{SimRat}, we know that
\begin{equation}\label{urat}
\lim_{\nri}\frac{U_{n+1}(x)}{U_n(x)}=x+\sqrt{x^2-1},\qquad x\not\in[-1,1].
\end{equation}
If we apply Corollary \ref{wall} and \eqref{urat} with $x=\frac{z+1}{2\rho\sqrt{z}}$ we conclude
\[
\lim_{\nri}\frac{A_n(z)}{B_n(z)}=\frac{2\alpha}{1-z+\sqrt{(z+1)^2-4\rho^2z}},\qquad\qquad |z|<1,
\]
which agrees with the formula given for $f$ in \cite[Section 1.6]{OPUC1}.  We can also use Theorem \ref{geronclosed} to deduce the ratio asymptotic behavior of the orthonormal Geronimus polynomials.  If we apply the formula from Theorem \ref{geronclosed} and \eqref{urat} with $x=\frac{z+1}{2\rho\sqrt{z}}$, we see that
\[
\lim_{\nri}\frac{\varphi_{n+1}(z;\mu_{\alpha})}{\varphi_n(z;\mu_{\alpha})}=\frac{z+1+\sqrt{(z+1)^2-4\rho^2z}}{2\rho},\qquad z\not\in\supp(\mu_{\alpha}).
\]
This result is not new and follows from the stronger results in \cite[Theorem 1]{BHMD}, but Theorem \ref{geronclosed} provides us with an easy proof.

Theorem \ref{geronclosed} also provides a new proof of the following fact, which appears in \cite[Equation 5]{Demeyer}.  To state it, we recall the polynomials $\{T_n\}_{n\in\bbN}$ from Section \ref{chebsec}.

\begin{corollary}\label{pell}
The pair $(X,Y)=(T_n(z),U_{n-1}(z))$ solves the Pell equation
\[
X^2-(z^2-1)Y^2=1.
\]
\end{corollary}

\begin{proof}
We recall \cite[Proposition 3.2.2]{OPUC1}, which tells us that for any measure $\mu$ with Verblunsky coefficients $\{\alpha_j\}_{j=0}^{\infty}$ it holds that
\[
\Psi_n^*(z;\mu)\Phi_n(z;\mu)+\Phi_n^*(z;\mu)\Psi_n(z;\mu)=2z^n\prod_{j=0}^{n-1}(1-|\alpha_j|^2)
\]
Applying this formula with $\mu=\mu_{\alpha}$, we find 
\[
U_n\left(\frac{z+1}{2\rho\sqrt{z}}\right)^2+U_{n-1}\left(\frac{z+1}{2\rho\sqrt{z}}\right)^2-\frac{z+1}{\rho\sqrt{z}}U_n\left(\frac{z+1}{2\rho\sqrt{z}}\right)U_{n-1}\left(\frac{z+1}{2\rho\sqrt{z}}\right)=1.
\]
Therefore, by invoking \eqref{TU} we find that for any $w\in\bbC$ it holds that
\begin{align*}
T_n(w)^2+U_{n-1}(w)^2&=(U_n(w)-wU_{n-1}(w))^2+U_{n-1}(w)^2=1+w^2U_{n-1}(w)^2
\end{align*}
as desired.
\end{proof}

One can also use Theorem \ref{geronclosed} to prove more substantial new results that require more detailed calculation and analysis.  The next section is devoted to just such a result, namely a universality result at the endpoint of the arc supporting the measure of orthogonality.

\section{Universality}\label{univ}

Let $\mu$ be a probability measure with infinite support on the unit circle.  The degree $n$ polynomial reproducing kernel $K_n(z,w;\mu)$ is given by
\[
K_n(z,w;\mu):=\sum_{m=0}^n\varphi_m(z;\mu)\overline{\varphi_m(w;\mu)}
\]
and is the reproducing kernel for the space of polynomials of degree at most $n$ in $L^2(\mu)$.  One is often interested in calculating the following limit (if it exists):
\begin{align}\label{klim}
\lim_{\nri}\frac{K_n(z_0+\sigma_1(n),z_0+\sigma_2(n);\mu)}{K_n(z_0,z_0;\mu)},
\end{align}
where $\sigma_j(n)\rightarrow0$ as $\nri$ in a specific way for $j=1,2$.   If this limit exists and is the same for a large class of measures $\mu$, then we call the corresponding result a \textit{universality} result.  Some universality results when the point $z_0$ is the endpoint of an interval supporting the measure of orthogonality can be found in \cite{Danka,KV,LubEdge,LubEdge2}, but all of these results assume that the measure is supported on a compact subset of the real line.  Our main result in this section is Theorem \ref{uend}, which considers measures supported on an arc of the unit circle.   Before we can state it, we need to define some notation.  If $J_{s}$ deontes the Bessel function of the first kind of order $s$, then we set
\[
\bbJ_{1/2}^*(a,b):=\begin{cases}
\frac{J_{1/2}(\sqrt{a})\sqrt{b}J_{1/2}'(\sqrt{b})-J_{1/2}(\sqrt{b})\sqrt{a}J_{1/2}'(\sqrt{a})}{2a^{1/4}b^{1/4}(a-b)},\qquad & a\neq b\\
\frac{1}{4\sqrt{a}}\left(J_{1/2}^2(\sqrt{a})-J_{3/2}(\sqrt{a})J_{-1/2}(\sqrt{a})\right) & a=b
\end{cases}
\]
as in \cite{Danka,LubEdge,LubEdge2}.  As noted in \cite{LubEdge}, the function $\bbJ_{1/2}^*$ is entire.  Now we can state our main result about universality after recalling the notion of regularity from Section \ref{reg}.  For the remainder of this section, we identify the unit circle with the interval $[0,2\pi)$.

\begin{theorem}\label{uend}
Fix $\alpha\in(-1,0)$ and let $\mu$ be a probability measure on the unit circle of the form $h(\theta)w(\theta)\frac{d\theta}{2\pi}+d\tilde{\mu}$ where
\[
\supp(\tilde{\mu})\subseteq[2\arcsin(|\alpha|)+\tilde{\epsilon},2\pi-2\arcsin(|\alpha|)]
\]
for some $\tilde{\epsilon}>0$ and
\[
w(\theta)=\begin{cases}
\frac{\sqrt{1-\alpha^2-\cos^2(\theta/2)}}{(1+\alpha)\sin(\theta/2)}\qquad\qquad & 2\arcsin(|\alpha|)<\theta<2\pi-2\arcsin(|\alpha|)\\
0 & o.w.
\end{cases}
\]
where $h(\theta)$ is continuous at $2\arcsin(|\alpha|)$ and $h(2\arcsin(|\alpha|))>0$.  Assume also that $\mu$ is regular.  Then uniformly for $a,b$ in compact subsets of the complex plane, it holds that
\[
\lim_{\nri}\frac{K_n(e^{i(\tha-\frac{a}{n^2})},e^{i(\tha-\frac{b}{n^2})};\mu)}{K_n(e^{i\tha},e^{i\tha};\mu)}=\frac{\bbJ_{1/2}^*(\frac{\alpha a}{\rho},\frac{\alpha\bar{b}}{\rho})}{\bbJ_{1/2}^*(0,0)},
\]
where $\tha=2\arcsin(|\alpha|)$ and $\rho=\sqrt{1-\alpha^2}$.
\end{theorem}

\noindent\textit{Remark.}  Notice that the limiting kernel in Theorem \ref{uend} is the same as in the real line case from \cite[Theorem 1.4]{Danka}.

\smallskip

Before we proceed with the proof of Theorem \ref{uend}, we present a proof of the following fact about the kernel $\bbJ^*_{1/2}$.

\begin{prop}\label{jz}
The function $\bbJ^*_{1/2}(t,\bar{t})$ is non-vanishing as a function of $t\in\bbC$.
\end{prop}

\begin{proof}
When $t$ is real, we use the fact that
\[
J_{1/2}(z)=\frac{\sqrt{2}\sin(z)}{\sqrt{\pi z}},\qquad J_{-1/2}(z)=\frac{\sqrt{2}\cos(z)}{\sqrt{\pi z}},\qquad J_{3/2}(z)=\frac{\sqrt{2}(\sin(z)-z\cos(z))}{\sqrt{\pi z^3}}
\]
(see \cite[pages 16 $\&$ 17]{Korenev}) to see that
\[
\bbJ_{1/2}^*(t,t)=\frac{1}{2\pi t}\left(1-\frac{\sin(2\sqrt{t})}{2\sqrt{t}}\right),\qquad\qquad t\in\bbR,
\]
which is non-zero for all $t\in\bbR$.  Using similar formulas, we find that when $t\not\in\bbR$, we have
\begin{equation}\label{jt}
\bbJ_{1/2}^*(t,\bar{t})=\frac{\frac{\cos(\sqrt{\bar{t}})\sin(\sqrt{t})}{\sqrt{t}}-\frac{\cos(\sqrt{t})\sin(\sqrt{\bar{t}})}{\sqrt{\bar{t}}}}{\pi(t-\bar{t})}=\frac{\Imag\frac{\cos(\sqrt{\bar{t}})\sin(\sqrt{t})}{\sqrt{t}}}{\pi\Imag t}.
\end{equation}
Since we are assuming $t\not\in\bbR$, we may assume $\Real[\sqrt{t}]\neq0$.  We will show that \eqref{jt} is never zero when $\Real[\sqrt{t}]>0,\Imag[\sqrt{t}]>0$ and the other cases can be deduced by using the symmetry of this expression.

Suppose $\sqrt{t}=\frac{1}{2}(x+iy)$ and $\sqrt{\bar{t}}=\frac{1}{2}(x-iy)$ (our choice of $\sqrt{\bar{t}}$ does not matter because the cosine function and the sinc function are both even).  Using basic trigonometric identities, we can rewrite the numerator of \eqref{jt} as
\begin{equation}\label{jt2}
\Imag\left[\frac{\sin(x)+i\sinh(y)}{x+iy}\right]=\frac{x\sinh(y)-y\sin(x)}{x^2+y^2}
\end{equation}
This is zero when $y=0$ and the first partial derivatives of both the numerator and denominator are positive when $x$ is positive.  This shows that \eqref{jt2} is positive when $x$ and $y$ are positive.  Similar calculations for negative values of $x$ or $y$ show \eqref{jt} is non-zero when $t\not\in\bbR$.  
\end{proof}

The proof of Theorem \ref{uend} will follow the method pioneered by Lubinsky, which consists of first proving the result in one particular case (when $h\equiv1$) and then using localization techniques and the assumed regularity of the measure to prove the more general case (see \cite{LubEdge2}).

\subsection{A Model Case}\label{model}

Fix $\alpha\in(-1,0)$ and define $\tha=2\arcsin(|\alpha|)$.  The measure $\mu_{\alpha}$ from Section \ref{geronsec} is of the form given in Theorem \ref{uend} with $h\equiv1$ (see \cite[Section 1.6]{OPUC1}). Let us write $z=e^{iw}$, where we allow $w$ to be complex.  Since $U_n$ is even or odd (depending on the parity of $n$) our choice of $\sqrt{z}$ will not effect our calculations, so we will write $\sqrt{z}=e^{iw/2}$.  Theorem \ref{geronclosed} then gives
\begin{align}
\label{w}\varphi_n(z;\mu_{\alpha})&=e^{inw/2}\left(U_n\left(\frac{\cos(w/2)}{\rho}\right)-\frac{1+\bar{\alpha}}{\rho e^{iw/2}}U_{n-1}\left(\frac{\cos(w/2)}{\rho}\right)\right)\\
\label{starw}\varphi_n^*(z;\mu_{\alpha})&=e^{inw/2}\left(U_n\left(\frac{\cos(w/2)}{\rho}\right)-\frac{e^{iw/2}(1+\alpha)}{\rho}U_{n-1}\left(\frac{\cos(w/2)}{\rho}\right)\right)
\end{align}
We will apply (\ref{w}) and (\ref{starw}) with $w=\tha+t/n^2$ for various values of $t$.  We begin with the following lemma.

\begin{lemma}\label{knormal}
The collection of functions
\[
\left\{\frac{K_n(e^{i(\tha-\frac{a}{n^2})},e^{i(\tha-\frac{\bar{b}}{n^2})};\mu_{\alpha})}{n^3}\right\}_{n\in\bbN}
\]
is a normal family on $\bbC^2$ in the variables $a$ and $b$.
\end{lemma}

\begin{proof}
By Montel's Theorem and the Cauchy-Schwarz inequality, it suffices to show that the collection
\[
\left\{\frac{K_n(e^{i(\tha-\frac{a}{n^2})},e^{i(\tha-\frac{a}{n^2})};\mu_{\alpha})}{n^3}\right\}_{n\in\bbN}
\]
is uniformly bounded in compact subsets of $\bbC$ (as a function of $a$).  To do so, we use \eqref{chebform2} to see that
\begin{align*}
&U_n\left(\frac{1}{\rho}\cos\left(\frac{\tha}{2}-\frac{a}{2n^2}\right)\right)=U_n\left(1-\frac{\alpha a}{2\rho n^2}+o(n^{-2})\right)=O(n)
\end{align*}
as $\nri$ and hence (\ref{w}) implies $|\varphi_n(e^{i(\tha-\frac{a}{n^2})};\mu_{\alpha})|=O(n)$ as $\nri$ uniformly for $a$ in compact subsets of $\bbC$.  It follows that
\[
K_n(e^{i(\tha-\frac{a}{n^2})},e^{i(\tha-\frac{a}{n^2})};\mu_{\alpha})=O\left(\sum_{m=1}^{n}m^2\right)=O(n^3)
\]
uniformly for $a$ in compact subsets of $\bbC$.  This is the desired conclusion.
\end{proof}

To prove Theorem \ref{uend} in the case $\mu=\mu_{\alpha}$ for $\alpha\in(-1,0)$, we apply the CD formula with $a\neq\bar{b}$ (see  \cite[Theorem 2.2.7]{OPUC1}).  Using Theorem \ref{geronclosed}, we find
\begin{align*}
&K_n(e^{i(\tha-\frac{a}{n^2})},e^{i(\tha-\frac{b}{n^2})};\mu_{\alpha})=(1-e^{i(\bar{b}-a)/n^2})^{-1}\\
&\quad\times\bigg[e^{i(\bar{b}-a)\frac{n+1}{2n^2}}\left(U_{n+1}\left(\frac{\cos\left(\frac{\tha}{2}-\frac{a}{2n^2}\right)}{\rho}\right)-\frac{e^{i\left(\frac{\tha}{2}-\frac{a}{2n^2}\right)}(1+\alpha)}{\rho }U_{n}\left(\frac{\cos\left(\frac{\tha}{2}-\frac{a}{2n^2}\right)}{\rho}\right)\right)\\
&\quad\left(U_{n+1}\left(\frac{\cos\left(\frac{\tha}{2}-\frac{\bar{b}}{2n^2}\right)}{\rho}\right)-\frac{e^{-i\left(\frac{\tha}{2}-\frac{\bar{b}}{2n^2}\right)}(1+\alpha)}{\rho }U_{n}\left(\frac{\cos\left(\frac{\tha}{2}-\frac{\bar{b}}{2n^2}\right)}{\rho}\right)\right)-\\
&\quad e^{i(\bar{b}-a)\frac{n+1}{2n^2}}\left(U_{n+1}\left(\frac{\cos\left(\frac{\tha}{2}-\frac{a}{2n^2}\right)}{\rho}\right)-\frac{(1+\alpha)}{\rho e^{i\left(\frac{\tha}{2}-\frac{a}{2n^2}\right)}}U_{n}\left(\frac{\cos\left(\frac{\tha}{2}-\frac{a}{2n^2}\right)}{\rho}\right)\right)\\
&\quad\left(U_{n+1}\left(\frac{\cos\left(\frac{\tha}{2}-\frac{\bar{b}}{2n^2}\right)}{\rho}\right)-\frac{(1+\alpha)}{\rho e^{-i\left(\frac{\tha}{2}-\frac{\bar{b}}{2n^2}\right)}}U_{n}\left(\frac{\cos\left(\frac{\tha}{2}-\frac{\bar{b}}{2n^2}\right)}{\rho}\right)\right)\bigg]
\end{align*}
\begin{align}
\nonumber&=\frac{2ie^{i(\bar{b}-a)\frac{n+1}{2n^2}}(1+\alpha)}{\rho(1-e^{i(\bar{b}-a)/n^2})}\times\\
\nonumber&\quad\bigg[-U_{n+1}\left(\frac{\cos\left(\frac{\tha}{2}-\frac{\bar{b}}{2n^2}\right)}{\rho}\right)U_{n}\left(\frac{\cos\left(\frac{\tha}{2}-\frac{a}{2n^2}\right)}{\rho}\right)\sin\left(\frac{\tha}{2}-\frac{a}{2n^2}\right)\\
\label{threeterm}&\qquad+U_{n+1}\left(\frac{\cos\left(\frac{\tha}{2}-\frac{a}{2n^2}\right)}{\rho}\right)U_{n}\left(\frac{\cos\left(\frac{\tha}{2}-\frac{\bar{b}}{2n^2}\right)}{\rho}\right)\sin\left(\frac{\tha}{2}-\frac{\bar{b}}{2n^2}\right)\\
\nonumber&\qquad+U_{n}\left(\frac{\cos\left(\frac{\tha}{2}-\frac{\bar{b}}{2n^2}\right)}{\rho}\right)U_{n}\left(\frac{\cos\left(\frac{\tha}{2}-\frac{a}{2n^2}\right)}{\rho}\right)\frac{1+\alpha}{\rho}\sin\left(\frac{\bar{b}-a}{2n^2}\right)\bigg]
\end{align}

Using basic angle addition formulas, we find
\begin{align*}
\frac{1}{\rho}\cos\left(\frac{\tha}{2}+\frac{t}{2n^2}\right)&=\cos\left(\frac{t}{2n^2}\right)+\frac{\alpha}{\rho}\sin\left(\frac{t}{2n^2}\right)=1+\frac{\alpha t}{2\rho n^2}+O(n^{-4})\\
\sin\left(\frac{\tha}{2}+\frac{t}{2n^2}\right)&=-\alpha\cos\left(\frac{t}{2n^2}\right)+\rho\sin\left(\frac{t}{2n^2}\right)=-\alpha+\frac{t\rho}{2n^2}+O(n^{-4})
\end{align*}
If $\mu^*$ is the measure of orthogonality for the polynomials $\{U_n\}_{n\geq0}$, then
\[
K_n(x,y;\mu^*)=\frac{\overline{U_{n}(y)}U_{n+1}(x)-U_{n}(x)\overline{U_{n+1}(y)}}{2(x-\bar{y})}
\]
(see \cite[Section 3]{SimonCD}).  Letting $x=\frac{1}{\rho}\cos\left(\frac{\tha}{2}-\frac{a}{2n^2}\right)$ and $y=\frac{1}{\rho}\cos\left(\frac{\tha}{2}-\frac{b}{2n^2}\right)$, we find that the first two terms in (\ref{threeterm}) can be rewritten
\begin{align*}
2\alpha\left(\frac{\alpha(a-\bar{b})}{2\rho n^2}\right)K_n\left(1-\frac{\alpha a(1+o(1))}{2\rho n^2},1-\frac{\alpha b(1+o(1))}{2\rho n^2};\mu^*\right)(1+o(1))
\end{align*}
as $\nri$.  By \cite[Theorem 1.4]{Danka}, we see that we can rewrite this as
\begin{align}\label{k11}
2\alpha\left(\frac{\alpha(a-\bar{b})}{2\rho n^2}\right)K_n\left(1,1;\mu^*\right)(1+o(1))\frac{\bbJ_{1/2}^*(\frac{\alpha a}{\rho},\frac{\alpha\bar{b}}{\rho})}{\bbJ_{1/2}^*(0,0)}
\end{align}
as $\nri$.  Using the fact that $U_n(1)=n+1$ (see \cite[page 37]{BE}) we find $K_n(1,1;\mu^*)=\frac{1}{6}(n+1)(n+2)(2n+3)$, so (\ref{k11}) simplifies to
\[
n\frac{\alpha^2(a-\bar{b})\bbJ_{1/2}^*(\frac{\alpha a}{\rho},\frac{\alpha\bar{b}}{\rho})}{3\rho\bbJ_{1/2}^*(0,0)}+o(n)
\]
as $\nri$.

To estimate the last term in (\ref{threeterm}), we use \eqref{chebform2} as in the proof of Lemma \ref{knormal} to see that
\begin{align*}
&U_n\left(\frac{1}{\rho}\cos\left(\frac{\tha}{2}+\frac{t}{2n^2}\right)\right)=U_n\left(1+\frac{\alpha t}{2\rho n^2}+o(n^{-2})\right)=O(n)
\end{align*}
as $\nri$.  Since $\sin\left(\frac{\bar{b}-a}{2n^2}\right)=O(n^{-2})$ as $\nri$, we see that the last term in (\ref{threeterm}) is $O(1)$ as $\nri$.
Combining all that we have learned so far yields
\begin{align}\label{kform}
&K_n(e^{i(\tha-\frac{a}{n^2})},e^{i(\tha-\frac{b}{n^2})};\mu_{\alpha})=o(n^3)+n^3\left(\frac{2(1+\alpha)\alpha^2\bbJ_{1/2}^*(\frac{\alpha a}{\rho},\frac{\alpha\bar{b}}{\rho})}{3\rho^2\bbJ_{1/2}^*(0,0)}\right)
\end{align}
as $\nri$ when $a\neq\bar{b}$.  By continuity and Lemma \ref{knormal} we may extend this formula to the case $a=\bar{b}$ and deduce that the error term can be estimated uniformly for $a$ and $b$ in compact subsets of $\bbC$.  Setting $a=b=0$, we find
\[
K_n(e^{i\tha},e^{i\tha};\mu_{\alpha})=o(n^3)+n^3\left(\frac{2(1+\alpha)\alpha^2}{3\rho^2}\right)\]
as $\nri$ (see also \cite[Theorem 1.2]{DT}).  We have thus proven
\[
\lim_{\nri}\frac{K_n(e^{i(\tha-\frac{a}{n^2})},e^{i(\tha-\frac{b}{n^2})};\mu_{\alpha})}{K_n(e^{i\tha},e^{i\tha};\mu_{\alpha})}=\frac{\bbJ_{1/2}^*(\frac{\alpha a}{\rho},\frac{\alpha\bar{b}}{\rho})}{\bbJ_{1/2}^*(0,0)},
\]
where the convergence is uniform for $a$ and $b$ in compact subsets of $\bbC$.  This proves the desired result in the case $\mu=\mu_{\alpha}$ when $\alpha\in(-1,0)$.

\subsection{The General Case}\label{general}

To prove the general case, we will use the following theorem, which is due to Bourgade and appears in a more general form as \cite[Theorem 3.10]{Bourgade}.

\begin{theorem}[\cite{Bourgade}]\label{bourgade310}
Let $\mu$ be as in Theorem \ref{uend} and let $\mu_{\alpha}$ be as in Section \ref{model}.  If
\begin{equation}\label{rlimsup}
\lim_{r\rightarrow0^+}\limsup_{\nri}\frac{K_{n}(e^{i(\tha-\frac{t}{n^2})},e^{i(\tha-\frac{t}{n^2})};\mu_{\alpha})}{K_{n-\left\lfloor rn\right\rfloor}(e^{i(\tha-\frac{t}{n^2})},e^{i(\tha-\frac{t}{n^2})};\mu_{\alpha})}=1
\end{equation}
uniformly for $t$ in compact subsets of $\bbR$, then uniformly for $a$ and $b$ in compact subsets of $\bbR$, it holds that
\[
\lim_{\nri}\left|\frac{h(\tha)K_{n}(e^{i(\tha-\frac{a}{n^2})},e^{i(\tha-\frac{b}{n^2})};\mu)-K_{n}(e^{i(\tha-\frac{a}{n^2})},e^{i(\tha-\frac{b}{n^2})};\mu_{\alpha})}{K_{n}(e^{i(\tha-\frac{a}{n^2})},e^{i(\tha-\frac{a}{n^2})};\mu)}\right|=0
\]
\end{theorem}

\noindent\textit{Remark.}  Note that \cite[Theorem 3.10]{Bourgade} includes a mutual regularity condition for the two measures in question, but the regularity of $\mu$ and $\mu_{\alpha}$ immediately implies that this condition is satisfied.


\smallskip

The fact that $\mu_{\alpha}$ satisfies the condition \eqref{rlimsup} is a direct consequence of the calculations in Section \ref{model}.  All that remains then is to show that the conclusion of Theorem \ref{bourgade310} gives us the conclusion that we want.  For this purpose, the following lemma is essential.

\begin{lemma}\label{lambdarat}
Let $\mu$ be as in the statement of Theorem \ref{uend} and let $\mu_{\alpha}$ be as in Section \ref{model}.  For any $a\in\bbC$ it holds that
\[
\lim_{\nri}\frac{K_n(e^{i(\tha-\frac{a}{n^2})},e^{i(\tha-\frac{a}{n^2})};\mu_{\alpha})}{K_n(e^{i(\tha-\frac{a}{n^2})},e^{i(\tha-\frac{a}{n^2})};\mu)}=h(\tha)
\]
and the convergence is uniform for $a$ in compact subsets of $\bbC$.
\end{lemma}

The proof of Lemma \ref{lambdarat} is very much analogous to the proof of \cite[Lemma 2.8]{SimUniv}, so we will not present the details here.  It is based on Christoffel functions and relies heavily on ideas from the proof of \cite[Theorem 7]{MNT}.

\begin{corollary}\label{krat}
Let $\mu$ be as in the statement of Theorem \ref{uend}.  Then
\begin{equation}\label{krat1}
\lim_{\nri}\frac{K_{n}(e^{i\tha},e^{i\tha};\mu)}{K_{n}(e^{i(\tha-\frac{a}{n^2})},e^{i(\tha-\frac{a}{n^2})};\mu)}=\frac{\bbJ_{1/2}^*(0,0)}{\bbJ_{1/2}^*(\frac{\alpha a}{\rho},\frac{\alpha\bar{a}}{\rho})}>0
\end{equation}
and the convergence is uniform on compact subsets of $\bbC$.  Furthermore, the collection
\[
\left\{\frac{K_{n}(e^{i(\tha-\frac{a}{n^2})},e^{i(\tha-\frac{\bar{b}}{n^2})};\mu)}{K_{n}(e^{i\tha},e^{i\tha};\mu)}\right\}_{n\in\bbN}
\]
is a normal family on $\bbC^2$ in the variables $a$ and $b$.
\end{corollary}

\noindent\textit{Remark.}  By Proposition \ref{jz}, the right-hand side of \eqref{krat1} is a well-defined positive real number.

\begin{proof}
The limit \eqref{krat1} follows from Lemma \ref{lambdarat} and the fact that the limit holds when $\mu=\mu_{\alpha}$.  The statement about normality follows from Montel's theorem, the Cauchy-Schwarz inequality, and the uniformity in the limit \eqref{krat1}.
\end{proof}

\begin{proof}[Proof of Theorem \ref{uend}]
We have already seen that we may apply Theorem \ref{bourgade310}.  By applying Corollary \ref{krat}, we may rewrite the conclusion of Theorem \ref{bourgade310} as
\[
\lim_{\nri}\left|\frac{K_{n}(e^{i(\tha-\frac{a}{n^2})},e^{i(\tha-\frac{b}{n^2})};\mu)}{K_{n}(e^{i\tha},e^{i\tha};\mu)}-\frac{K_{n}(e^{i(\tha-\frac{a}{n^2})},e^{i(\tha-\frac{b}{n^2})};\mu_{\alpha})}{h(\tha)K_{n}(e^{i\tha},e^{i\tha};\mu)}\right|=0
\]
By applying Lemma \ref{lambdarat}, we may rewrite this as
\[
\lim_{\nri}\left|\frac{K_{n}(e^{i(\tha-\frac{a}{n^2})},e^{i(\tha-\frac{b}{n^2})};\mu)}{K_{n}(e^{i\tha},e^{i\tha};\mu)}-\frac{K_{n}(e^{i(\tha-\frac{a}{n^2})},e^{i(\tha-\frac{b}{n^2})};\mu_{\alpha})}{K_{n}(e^{i\tha},e^{i\tha};\mu_{\alpha})}\right|=0,
\]
and hence when $a$ and $b$ are real, the desired convergence follows from the calculations in Section \ref{model}.  The desired uniform convergence on compact subsets of $\bbC^2$ follows from the statement about normal families in Corollary \ref{krat}.
\end{proof}



\end{document}